\documentclass[10pt,twoside,reqno]{amsart}
\usepackage{amsmath}
\usepackage{amssymb}
\usepackage{mathrsfs}
\usepackage[mathcal]{eucal}
\usepackage[vcentermath]{youngtab}
\usepackage[shortalphabetic, initials]{amsrefs}
\newcount\NBcount

\long\def\NB#1{\advance\NBcount by 1{\Tiny\boxed{\number\NBcount}}%
  \marginpar{\raggedright\Small [\number\NBcount] #1}}

\newtheorem{Thm}{Theorem}
\newtheorem{Prop}{Proposition}
\newtheorem{Lem}[Prop]{Lemma}
\newtheorem{Cor}[Prop]{Corollary}

\theoremstyle{definition}
\newtheorem{Def}{Definition}
\newtheorem{Ex}[Def]{Example}
\newtheorem{Rem}[Def]{Remark}

\def\le{\leqslant}
\def\ge{\geqslant}
\def\C{{\mathbb C}}
\def\R{{\mathbb R}}
\def\Q{{\mathbb Q}}
\def\Z{{\mathbb Z}}
\def\N{{\mathbb N}}
\def\codim{\operatorname{codim}}
\let\parasymbol=\S
\def\secref#1{\parasymbol\ref{#1}}
\def\S{\varSigma}

\def\Diff{\operatorname{Diff}}

\def\l{\lambda}

\def\L{\varLambda}
\def\e{\varepsilon}

\def\^#1{\widehat{#1}}

\def\:{\colon}
\def\O{{\mathscr O(\C^d,0)}}

\def\~#1{\tilde{#1}}

\def\Aut{\operatorname{Aut}}

\begin{document}

\title[Dynamics of intersections]{Local dynamics of intersections:\\ V. I. Arnold's theorem revisited}

\author{Anna Leah Seigal, Sergei Yakovenko}
\date{August 2, 2012}

\address{
Cambridge University\\
UK\\
\textup{\texttt{anna.seigal1@gmail.com}}
}

\address{
Weizmann Institute of Science\\
Israel\\
\textup{\texttt{sergei.yakovenko@weizmann.ac.il}}
}

\begin{abstract}
V. I. Arnold proved in 1991 (published in 1993) that the intersection multiplicity between two germs of analytic subvarieties at a fixed point of a holomorphic invertible self-map remains bounded when one of the germs is dragged by iterations of the self map. The proof is based on the Skolem-Mahler-Lech theorem on zeros in recurrent sequences. 

We give a different proof, based on the Noetherianity of certain algebras, which allows to generalize the Arnold's theorem for local actions of arbitrary finitely generated commutative groups, both with discrete and infinitesimal generators. Simple examples show that for non-commutative groups the analogous assertion fails.
\end{abstract} 

\maketitle

\section{Complexity of iterated intersections}

\subsection{Complexity growth by iterations}
Let $F:M\to M$ be a smooth self-map of a differentiable manifold $M$ into itself. For a pair $X,Y\subseteq M$ of two subvarieties of complimentary dimensions $\dim X+\dim Y=\dim M$, one can consider the intersection $F^n(X)\cap Y$ between the image of $X$ by the $n$th iteration $F^n=F\circ\cdots \circ F$ and $Y$. If this intersection consists of finitely many points, then the growth rate of the integer sequence $m_n=m_n(F,X,Y)=\#\{F^n(X)\cap Y\}\in\N$ for $n=0,1,2,\dots$ is an important indicator of the mixing properties of the dynamical system. In particular, this construction allows to measure the growth of the number of $n$-periodic points (the fixed points of $F^n$). To that end, one has to replace $M$ by its square $M\times M$ and $F$ by the product $F\times\operatorname{id}$, choosing $X$ and $Y$ to be the same subvariety, the diagonal $M\subseteq M\times M$. Then the intersection after $n$ iterations will coincide with the set of $n$-periodic points.  

The problem of estimating the possible growth rate of the sequence $m_n$ for various classes of dynamical systems was discussed by V.~I.~Arnold in a series of papers in the beginning of 1990ies. If $M$ is an algebraic variety and $F$ an algebraic map, then the sequence $m_n$ grows no faster than exponentially by the B\'ezout theorem. In \cites{a-smooth,a-boston} it was shown that the same exponential growth rate is characteristic also for a \emph{generic} smooth map $F$, with the genericity conditions explicitly spelled out. The parallel exponential upper bound for the growth of the number of periodic points is given by the Artin--Mazur theorem \cite{katok}. The results remain true if in the definition of the numbers $m_n$ the intersection points are counted with their \emph{multiplicity}, see below.

However, if we replace the global problem with a local one, the situation changes radically. Let $F:(\R^d,0)\to(\R^d,0)$ be the germ of an invertible $C^1$-smooth self-map and the origin $0$ is an isolated fixed point for all iterations $F^n:(\R^d,0)\to(\R^d,0)$.  Then the multiplicity of $F^n$ at the origin, defined as the topological index of the vector field $v_n(x)=F^n(x)-x$, is uniformly bounded as $n\to\infty$, as was discovered by M.~Shub and D.~Sullivan \cite{sul-shub}.

V.~Arnold in \cite{a-miln} generalized this result for the local complexity of holomorphic iterations in the general setting. Let $X,Y\subseteq(\C^d,0)$ be two germs of analytic submanifolds of complimentary dimensions. Since $F(0)=0$, the origin is always a common point of the iterates $F^n(X)$ and $Y$ for any $n$  and one may compute the \emph{intersection multiplicity} $\mu_n=\mu_n(F,X,Y)>0$ between $F^n(X)$ and $Y$ (the precise definition is given below at \secref{sec:mult}). If the intersection $F^n(X)\cap Y$ consists of the single (isolated) point at the origin, the corresponding multiplicity is a well defined \emph{finite} natural number, $0<\mu_n=\mu_n(F,X,Y)<+\infty$.

Arnold's theorem \cite{a-miln}*{Theorem~1} claims that if $F:(\C^d,0)\to(\C^d,0)$ is the germ of a holomorphic \emph{invertible} self-map and $X,Y\subset(\C^d,0)$ a pair of analytic subvarieties of complimentary dimensions, such that the origin is an isolated intersection of $F^n(X)$ and $Y$ for all $n=1,2,\dots$, then the multiplicity of this intersection $\mu_n(F,X,Y)$ stays bounded as $n\to\infty$.

The assumption of invertibility is crucial: the germ $F:(\C^2,0)\to(\C^2,0)$, $F(x,y)=(x,y^2)$, after $n$ iterations maps the diagonal $X=\{y=x\}$ into the curve $F^n(X)=\{y=x^{2^n}\}$ which intersects the axis $Y=\{y=0\}$ with the multiplicity $2^n$ which grows exponentially as $n\to\infty$ \cite{a-miln}. Very recently W. Gignac constructed in \cite{gignac} an example of a non-invertible (strongly contracting) \emph{polynomial} self-map, for which the multiplicity of contact between two analytic curves $F^n(X)$ and $Y$ grows faster than any prescribed sequence.

In contrast with rather transparent arguments of Shub and Sullivan, the proof by Arnold rests heavily on the Skolem--Mahler--Lech theorem, a very deep fact from the number theory \cite{skolem}, which gives an explicit description of integer roots of quasipolynomials in one variable (see below for the precise definition; usually it is formulated for zero terms in recurrent sequences).

\subsection{Nonisolated intersections, flows and other general group actions}
There is a number of natural questions which are motivated by the above results of Shub--Sullivan and Arnold.

Both  theorems fail to address the case where some of the intersections $F^n(X)\cap Y$ are non-isolated (resp., where some of the iterates $F^n$ have a non-isolated fixed point at the origin). This seemingly technical restriction is known to cause serious problems in certain situations, cf.~with \cite{bn-11} and references therein. In this case the multiplicity of intersection becomes infinite, $\mu_n=\infty$ for some values of $n$, but one may still ask whether the sequence of \emph{finite} intersection multiplicities remains bounded.

Besides, from the dynamical point of view it is natural in parallel with the discrete time dynamical system of iterates $F^n$, $n\in \N$, to consider also the continuous time one-parametric groups $\{F^t\}$ of holomorphic germs with real ($t\in \R$) or complex ($t\in\C$) time. Such groups arise as collections of the flow maps (in short, the \emph{flow} of holomorphic vector fields on $(\C^d,0)$ with a singular point at the origin. By definition of the flow, it satisfies the ordinary differential equation $\left.\frac {\mathrm d}{\mathrm dt}F^t\right|_{t=0}=v$, where $v$ is the germ of the vector field, called the \emph{infinitesimal generator of the group} $\{F^t\}$.

Given the one-parametric group and any two germs of analytic manifolds $X,Y\subseteq(\C^n,0)$ of complementary dimensions, one can define the (extended)-integer-valued function $\mu_t=\mu_t(v,X,Y)$ as the intersection multiplicity between the image $F^t(X)$ and $Y$ (as before, we assume that $\mu_t=+\infty$ if the intersection is non-isolated). One could expect that the \emph{finite} values of this function should remain bounded. Note that in this case the nondegeneracy assumption follows automatically from the reversibility of the flow. On the other hand, unlike in the discrete time case, there is no Skolem theorem for the real or complex roots of quasipolynomials, thus the study should be based on different ideas.

Next, one can easily reformulate the question in different regularity assumptions, e.g., for $C^\infty$-smooth maps (resp., flows), or for formal self-maps (formal flows). In all cases the notion of multiplicity of intersection (eventually, infinite) is well defined, making the questions meaningful.

One can also reformulate the question for ``generalized dynamical systems''. Instead of just one invertible map, which induces a holomorphic action of the group $\Z$ (or one flow which induces a holomorphic action of the groups $\R$ or $\C$) we may ask the question about actions of more general groups. For instance, one can consider several holomorphic invertible self-maps $F_i:(\C^d,0)\to(\C^d,0)$, $i=1,\dots,p$ and the group $G$ generated by them using compositions and inversions. Then for any element $g\in G$ one can consider the image $g(X)$ and its intersection multiplicity $\mu(g)=\mu(g^{-1}(X),Y)$. It is interesting to study the distribution of finite values of the function $\mu=\mu_{XY}:G\to\N^*=\N\cup\{+\infty\}$, $g\mapsto\mu(g)$. The answer obviously depends on the \emph{algebraic} nature of the group $G$.

Very often instead of just one local self-map (flow, finitely generated group) one has to deal with \emph{parametric families} of such objects, depending on additional parameters in a way that is regular enough (smooth, analytic, etc.). In this case the main question is how the intersection multiplicity bounds depend on these parameters, in particular, whether they are they locally uniformly bounded. If only finite multiplicity is allowed (as in the initial version of Arnold's theorem), then such uniformity would follow from simple semicontinuity arguments, however, in presence of nonisolated intersections the affirmative answer is not automatic anymore.

We were able to answer all these questions in the affirmative sense for the widest class of formal groups, assuming only the commutativity of the group $G$, provided that it is finitely generated.

\subsection{Definitions and terminology}
We shall formulate and prove the main results for actions of commutative groups on the ring of formal series. All other (analytic, smooth etc.) versions will follow as easy corollaries. For the sake of simplicity, we deal first with the non-parametric case; the parametric case requires some extra technical work which will be outlined in the Appendix below.

Recall some basic terminology and facts. Denote by $\C[[x]]=\C[[x_1,\dots,x_d]]$ the algebra of formal Taylor series in $d$ variables with complex coefficients. This is a \emph{local algebra} with the unique maximal ideal of germs without the free term, denoted by $\mathfrak m$. A \emph{formal self-map} $F$ is a $d$-tuple of formal series without the free term; under this assumption one can define the composition operator $F^*:\C[[x]]\to\C[[x]]$,  $f\mapsto f\circ F$ which is an algebra homomorphism (in particular, it is $\C$-linear). By definition, $F^*\mathfrak m=\mathfrak m$. This homomorphism is an isomorphism if $F$ is formally invertible (the matrix of linear terms has nonzero determinant). Invertible formal self-maps can be composed with each other, forming the group which we will denote by $\Diff[[\C^d,0]]=\Aut\C[[x]]$. In a similar way a formal vector field $v$ is a $d$-tuple of formal series $v=(v_1,\dots,v_d)$ without the free terms, $v_i\in\mathfrak m\subset\C[[x]]$, which acts on the algebra $\C[[x]]$ as a derivation, $V:\C[[x]]\to\C[[x]]$, $V\!f=\sum_{j=1}^d\frac{\partial f}{\partial x_j}\cdot v_j(x)$. The \emph{formal flow} of a formal vector field is a one-parametric subgroup $\{F^t:t\in\C\}\subset\Diff[[\C^d,0]]$ which satisfies the differential equation $\left.\frac{\mathrm d}{\mathrm dt}F^t\right|_{t=0}=v$. One can show that such a subgroup always exists and the corresponding automorphism is given by the formal exponential series:
\begin{equation}\label{exp}
 {F^t}^*=\mathrm e^{tV}=\operatorname{id}+\,tV+\tfrac1{2!}\,t^2V^2+\cdots+\tfrac1{k!}\,t^k V^k+\cdots
\end{equation}
The series in right hand side converges termwise to an automorphism of $\C[[x]]$ for all complex values of $t\in\C$, associated with the corresponding one-parametric subgroup of formal self-maps $\{F^t\}$ \cite{thebook}. The corresponding formal self-maps will be denoted by $F^t=\mathrm e^{tv}$. 

All these constructions remain valid if we replace the algebra $\C[[x_1,\dots,x_d]]$ by the algebra $\mathscr O(\C^d,0)$ of holomorphic germs or the algebra $\mathscr E(\R^d,0)$ of germs of $C^\infty$-smooth real functions, and are naturally interrelated by associating with each germ (holomorphic or smooth) its formal Taylor series.

A \emph{formal subvariety} $X$ will be identified with a radical ideal $I=I_X\subset\C[[x]]$ in the same way the germ of an analytic subvariety $X$ can be identified with the (radical) ideal $I_X\subseteq\O$ of holomorphic germs vanishing on it. Since both rings are Noetherian, each ideal $I\subseteq\C[[x]]$ or $I\subseteq\O$ is always finitely generated: $I=\left<f_1,\dots,f_s\right>$. The algebra of $C^\infty$-smooth germs is not Noetherian, but we will only consider \emph{finitely generated} ideals in $\mathscr E(\R^d,0)$. The (germ of the) zero locus $X=V(I)$ of a finitely generated ideal $I$ is called a \emph{finitely presented local subvariety} in $(\R^d,0)$. For a formal self-map $F$ and a formal variety $X$ associated with an ideal $I$, the formal variety associated with the ideal $F^*I$ is called the preimage $F^{-1}(X)$ (in full agreement with the local analytic case, where the system of equations $\{(f_i\circ F)(x)=0\}$ means that $F(x)=y\in X=\{f_i(x)=0\}$).

For any two (formal, analytic or finitely presented) subvarieties $X,Y$ we define the multiplicity of their intersection as the codimension of the ideal $\left<I_X,I_Y\right>$ generated jointly by $I_X$ and $I_Y$:
\begin{equation*}
 \forall\ X,Y \quad \mu(X,Y)=\codim_\C\left<I_X,I_Y\right>\in\N^*=\N\cup\{+\infty\}.
\end{equation*}
This codimension may be a finite natural number or infinity. In an obvious way this definition can be generalized for the multiplicity of a multiple intersection $X_1,\dots,X_s$ of several (finitely presented in the $C^\infty$-smooth case) varieties.

A subgroup $G\subset\Diff[[\C^d,0]]$ is \emph{finitely generated}, if there exist finitely many invertible formal self-maps $F_1,\dots,F_p\in\Diff[[\C^d,0]]$ and formal vector fields $v_1,\dots,v_q$ such that every element of $G$ can be represented as a finite composition of the maps $F_i^{\pm1}$ and flows $\mathrm e^{t_j v_k}$. We will be mainly interested in the commutative case, where such composition can be described by a $r$-tuple of numbers, $r=p+q$, part of them integer, the rest complex (the cases $p=0$ or $q=0$ are not excluded).

A finitely generated group $G\subset\Diff[[\C^d,0]]$ is \emph{commutative} if and only if the generators commute between themselves, that is,
\begin{enumerate}
 \item $F_i\circ F_j=F_j\circ F_i$ for all $i,j=1,\dots,p$,
 \item $[v_i,v_j]=0$ for all $i,j=1,\dots,q$ (the formal Lie bracket),
 \item Each map $F_i$ preserves each of the vector fields $v_j$, $(F_i)_*v_j=v_j$ for all $i=1,\dots,p$, $j=1,\dots,q$.
\end{enumerate}
Each element of a commutative finitely generated group $G$ can be represented as
\begin{gather}
 g=F_1^{t_1}\circ \cdots \circ F_p^{t_p}\circ \mathrm e^{t_{p+1} v_1}\circ\cdots\circ \mathrm e^{t_{p+q}v_q},\label{t}
 \\
 \qquad t_1,\dots,t_p\in\Z,\ t_{p+1},\dots,t_{p+q}\in\C.\label{zc}
\end{gather}
In the group theoretic sense every commutative finitely generated subgroup of resp., $\Diff[[\C^d,0]]$ is isomorphic to a quotient of the ``free'' group $\Z^p\times\C^q$ by any identities that may occur (e.g., one of the generators $F_i$ may be periodic, $F^s=\operatorname{id}$ for a finite value of $s\in\N$). Conversely, the image of any group homomorphism (representation) $G:\Z^p\times\C^q\to\Diff[[\C^d,0]]\simeq\Aut\C[[x]]$ is an Abelian finitely generated subgroup of self-maps. We will identify the group $G$ with its image by the map (\ref{t},\ref{zc}) and write $g=G(t)$ for a generic element of $G$.

For any given pair of formal subvarieties $X,Y$ the intersection multiplicity function $\mu_{XY}(g)=\mu(g^{-1}(X),Y)$ pulls back as the function
\begin{equation}\label{mut}
 G^*\mu_{X,Y}:\Z^p\times\C^q\to\N^*,\qquad t\mapsto\mu\bigl(G(t)^{-1}(X),Y\bigr).
\end{equation}
%
%The linear terms of each self-map $F_i$, resp., vector field $V_j$, form a square $(d\times d)$-matrix over $\C$, whose eigenvalues form a subset in $\C$, called the \emph{spectrum} of the corresponding generator.

%\begin{Rem}\label{rem:embed}
%Any cyclic subgroup $G=\{F^t:t\in\Z\}\subset\Diff[[C^d,0]]$ of formal self-maps can be embedded in a formal flow $G^\C=\{F^t:t\in\C\}\subset\Diff[[\C^d,0]]$. For germs $F(x)=Ax+\cdots$ with the unipotent linearization part $A=\operatorname{id}+N$, $N$ a nilpotent $(d\times d)$-matrix, this is proved in \cite{thebook}*{Theorem 3.14, \parasymbol 3D}. The general case is obvious for formal self-maps in  Poincar\'e--Dulac normal form \cite{thebook}*{\parasymbol 4}. Such germs have the form $F(x)=D\circ F_\bullet$ with a diagonal linear self-map $D$ commuting with a formal unipotent self-map $F_\bullet$, and the flow has the form $F^t=\mathrm e^{tL}\circ F_\bullet^t$, where $L$ is the matrix logarithm of $A$ \cite{thebook}*{Lemma~3.11}.
%
%One can easily verify that for a commutative group with several ``discrete'' generators the corresponding flows can be constructed to be commuting again.
%
%The embeddability assertion (even for a single generator) is false (albeit for different reasons) for subgroups of $\Diff(\C^d,0)$ or $\Diff(\R^d,0)$. In the first case the obstruction is related to the so called Ecalle--Voronin modulus \cite{thebook}*{Chapter IV}, in the second,---with the impossibility of embedding orientation-reverting smooth self-maps into a real flow.
%\end{Rem}

\subsection{Principal results}

\begin{Thm}\label{thm:main}
Let $G\subset\Diff[[\C^d,0]]$ be a commutative finitely generated subgroup of formal self-maps.

Then for any two formal subvarieties $X,Y$ the multiplicity intersection function $\mu=\mu_{XY}:G\to\N^*$, $\mu(g)=\mu(g^{-1}(X),Y)\le +\infty$, has bounded finite values\textup: there exists a finite constant $m\in\N$ depending on $G$ and the two varieties $X,Y$, such that
\begin{equation}\label{bdm}
 \forall g\in G\ \mu(g)<+\infty\implies \mu(g)\le m.
\end{equation}
\end{Thm}

\begin{Cor}
For a commutative finitely generated subgroup $G\subset\Diff(\C^d,0)$ of germs of holomorphic self-maps and any two germs of analytic subvarieties $X,Y\subset(\C^d,0)$ the multiplicity of all \emph{isolated} intersections between $g^{-1}(X)$ and $Y$ is uniformly bounded.
\end{Cor}

\begin{proof}
The intersection of two holomorphic germs is (complex) isolated if and only if the intersection multiplicity is infinite.
\end{proof}

\begin{Cor}
The assertion of Theorem~\ref{thm:main} remains true for any finitely generated subgroup $G\subset\Diff(\R^d,0)$ of germs of $C^\infty$-smooth self-maps and any pair of finitely presented $C^\infty$-smooth subvarieties $X,Y\subseteq(\R^d,0)$.
\end{Cor}

\begin{proof}
By taking the Taylor series (eventually diverging) of the generators of the group and the equations defining the subvarieties $X,Y$, one reduces the situation to the purely formal case.

Note that the assertion of the theorem is much less informative in the $C^\infty$-smooth category. For instance, any germ of a closed set can be defined by a single $C^\infty$-smooth equation $f=0$, and hence very pathological intersection patterns can be designed. However, such an equation is (in the really bad cases) necessarily flat (having identically zero Taylor series), so all intersection multiplicities will be identically infinite for all elements of the group.
\end{proof}

The parametric case of commutative groups generated by self-maps and fields analytically depending on finitely many parameters, is discussed in the Appendix.

\subsection{Acknowledgments}
The first author gratefully acknowledges the support of the Kupcinet-Getz Internetional Summer Science School at the Weizmann Institute, during which the main part of the work was done. The second author is supported by the Israel Science Foundation grant 493/09. 
 
We thank all our friends and colleagues who discussed with us the work while it was in progress.  Our special gratitude goes to  Vladimir Berkovich, Gal Binyamini, Dmitry Novikov and Dmitry Gourevitch for their very insightful remarks.

\section{Intersection multiplicity and its algebraic computability}\label{sec:mult}
In this section we recall (and give a complete proof for readers' convenience) of the well known (folklore) fact that the codimension of a finitely generated ideal in $\C[[x]]$ is an ``algebraically computable'' function of its generators. This fact is one of the ingredients of Arnold's original proof \cite{a-miln}*{Lemma 1}.

%For any finite $m\in\N$ the quotient $\C[[x]]/\mathfrak m^{\ell+1}$, which is a finite dimensional linear algebra over $\C$, called the algebra of truncated polynomials. The natural projection is called truncation. Since automorphisms from $\Diff[[\C^d,0]]$ preserve $\mathfrak m$, they induce \emph{truncated action} on the quotient for any finite $\ell$.

\subsection{Algebraic decidability of the multiplicity}
Let $n\in\N$ be a fixed natural number. Consider all ideals in $\C[[x]]$ generated by $n$ elements $f_1,\dots,f_n$. Expanding each generator as a formal series,
\begin{equation*}
 f_i(x)=\sum_{i,\alpha}a_{i\alpha}x^\alpha, \qquad x=(x_1,\dots,x_d),\ \alpha\in\mathbb Z_+^d,\ a_{i\alpha}\in\C,
\end{equation*}
we see that the (infinitely many) Taylor coefficients $\{a_{i\alpha}\in\C:i=1,\dots, n,\ \alpha\in\Z^d_+\}$ form ``coordinates'' on the set of all ideals with $\le n$ generators. We will show that for any natural $m$ the condition that the ideal $I=\left<f_1,\dots,f_n\right>$ has codimension $\ge m$ is equivalent to a finite number of polynomial identities involving the Taylor coeffiients $a_{i\alpha}$.

\begin{Ex}
The condition that the codimension of $I$ is positive, is equivalent to the algebraic conditions $a_{i0}=0$, $i=1,\dots,n$ (no free terms). The condition that $\codim I\ge 2$ also can be expressed as the algebraic rank condition on the Jacobian matrix $J$ of the first order coefficients $a_{i\alpha}$ with $|\alpha|=1$ (vanishing of all minors of size $d$). In particular, when $n=d$, we have a single polynomial condition of vanishing the determinant, $$\operatorname{codim}I\ge 1\iff \det J=0,\qquad J=\|a_{i\alpha}\|_{i=1,\dots,n,\ |\alpha|=1}.$$
\end{Ex}

\begin{Lem}\label{lem:mu-alg}
For any finite $m$ the condition that $\codim I\ge m$ is equivalent to a finite number of algebraic conditions imposed on the Taylor coefficients $\{a_{i\alpha}\}$ of orders not exceeding $m$.
\end{Lem}

\begin{proof}
Consider the quotient algebra $\C_m[x]=\C[[x]]/\mathfrak m^{m+1}$ of \emph{truncated polynomials} of degree $\le m$, which is finite dimensional over $\C$, and the ideal $I_m=I\bmod \mathfrak m^{m+1}$ which is the quotient image of $I$ in $\C_m[x]$. The ideal $I_m$ is generated by truncation of the series $f_1,\dots,f_n$ to order $m$. The codimension of $I_m$ in $\C_m$ is the codimension of the image of the $\C$-linear map
$$
 (u_1,\dots,u_n)\mapsto u_1f_1+\cdots+u_n f_n,\qquad u_i,f_i\in\C_m[x].
$$
The matrix of this map (in the monomial bases in the source and the target spaces) consists of suitably placed coefficients $a_{i\alpha}$ with $|\alpha|\le m$. The condition that the image has codimension $\ge k$ for any natural $k$, is equivalent to vanishing of all minors of size $k+1$ and larger, which translates into finitely many polynomial conditions on the parameters $a_{i\alpha}$ with rational coefficients which define a radical ideal in the ring $\Q[a_{i\alpha}:|\alpha|\le m]$.

It remains to mention the general fact: if $I$ is an ideal of codimension less than $m$, then the $m$-th power of the maximal ideal $\mathfrak m^m$ automatically belongs to $I$ \cite{avg-1}*{\parasymbol 5.5, Ch.~I}. This means, that truncating the generators $f_i$ to their Taylor polynomials of order $\le m$ cannot affect the ideal, i.e., the conditions characterizing ideals of codimension $<m$ (and hence the complement, the ideals of codimension $\ge m$) cannot involve high order coefficients, being thus the same algebraic conditions as for the ring of the truncated polynomials.
\end{proof}

Using the terminology suggested by Arnold in \cite{arn-local} and further elaborated in \cite{thebook}*{\parasymbol 10}, one can say that computation of the codimension of an ideal in the local ring $\C[[x]]$ is an algebraically decidable problem.

\subsection{Infinite ascending chain of ideals}
Denote by $\mathfrak A=\mathfrak A_{dn}=\Q[\dots, a_{i\alpha},\dots]$ the ring of polynomials in the infinitely many variables $\{a_{i\alpha}:i=1,\dots,n,\ |\alpha|\ge 0\}$ (each element of $\mathfrak A$ may depend only on finitely many variables). The polynomial conditions described in Lemma~\ref{lem:mu-alg} define an increasing chain of radical ideals $\mathfrak I_m$ in this ring, generated by the corresponding polynomials:
\begin{equation}\label{chain}
 \mathfrak I_1\subsetneq\mathfrak I_2\subsetneq\cdots\subsetneq\mathfrak I_m\subsetneq\cdots\subsetneq \mathfrak A.
\end{equation}
The ``universal'' (depending only on the number of generators $n$ and the number of independent variables $d$) ideal $\mathfrak I_m$ is generated by the polynomial conditions guaranteeing that the codimension of the ``general'' ideal $I=\left<\sum a_{i\alpha}x^\alpha:i=1,\dots,n\right>$ is $\ge m$. Since the ring $\mathfrak A$ is \emph{not} Noetherian (the number of ``independent variables'' $a_{i\alpha}$ is infinite), the chain \eqref{chain} does not have to stabilize (and indeed does not, since there are ideals  of arbitrarily high but finite codimension).

\section{Quasipolynomials}
In this section we study how the elements $G(t)$, $t\in\Z^p\times\C^q$ of a commutative subgroup $G\subset\Diff[[\C^d,0]]$ depend on $t$. This result is also of the folklore nature, at least in the one-dimensional case.

\subsection{Quasipolynomials in several variables}
Let $\L\subset\C$ be a subset (eventually infinite) of the complex vector space $\C^r$.

\begin{Def}
A (complex) \emph{quasipolynomial in $r$ variables with the spectrum in $\L$} is a \emph{finite} sum of the form
\begin{equation}\label{qp}
 q(z)=\sum_{\l\in\L}\mathrm e^{\left<\l, z\right>}p_i(z),\qquad p_\l\in\C[z]=\C[z_1,\dots,z_r],\ z\in \C^r,
\end{equation}
where $\left<\lambda,z\right>=\sum_1^r \l_i z_i$.
\end{Def}
If $\L$ is a lattice, that is, $\L+\L\subseteq\L$ (in the Minkowski sense), then all quasipolynomials with the given spectrum $\L$ form a commutative algebra over $\C$. We will denote this algebra by $\C[\mathrm e^{z\L},z]$, $z=(z_1,\dots,z_r)$. 

%The representation \eqref{qp} is unique: the \emph{quasimonomials} $\mathrm e^{\left<\l, z\right>}z^\alpha$ corresponding to the different tuples $(\l,\alpha)\in\C^r\times\Z_+^r$, are linear independent.

By this definition, quasipolynomials are entire functions on $\C^r$. A quasipolynomial on $\Z^p\times\C^q$ is a by definition the restriction of a complex quasipolynomial on $\C^r$, $r=p+q$, on the the subset $\Z^p\times\C^q$. Note that in the discrete case each exponential can be changed by an integer multiple of $2\pi \mathrm i$, thus it is convenient to add $\pm2\pi\mathrm i$ to the generators of the corresponding lattice. 

Quasipolynomials in one variable naturally arise as solutions of linear ordinary differential equations with constant coefficients. The values of a quasipolynomial at all integer points $z\in\Z$ solve a linear recurrence with constant coefficients and vice versa.

\subsection{Noetherianity of the algebra of quasipolynomials}
We prove an obvious sufficient condition on the lattice $\L\subset\C^r$ guaranteeing that the corresponding algebra of quasipolynomials $\C[\mathrm e^{z\L},z]$ is Noetherian, that is, every ascending chain of ideals in this algebra stabilizes after finitely many steps.

\begin{Lem}\label{lem:noether}
If the lattice $\L\subseteq\C^r$ is finitely generated, then the algebra $\C[\mathrm e^{z\L},z]$ is Noetherian.
\end{Lem}

\begin{proof}
Assume that the vectors $\boldsymbol\l_1,\dots,\boldsymbol\l_s\in\C^r$ span the lattice $\L$ so that $\L\subseteq\sum_{i=1}^s\Z_+\boldsymbol\l_i$. Consider the algebra of polynomials $\C[w,z]=\C[w_1,\dots,w_s,z_1,\dots,z_r]$ and the map $\C[w,z]\to\C[\mathrm e^{z\L},z]$ which is obtained by extending the correspondence
$$
 z\mapsto z,\qquad w_i\mapsto \mathrm e^{\left<\boldsymbol\l_i,z\right>}
$$
as a ring homomorphism. By definition, each for each monomial of the form $\mathrm e^{\left<\l,z\right>}z^\alpha$ with $\l=\sum_1^s\beta_i\boldsymbol\l_i\in\C^r$ is the image of the monomial $w^\beta z^\alpha$, so that the above homomorphism is surjective. Since the ring $\C[w,z]$ is Noetherian, by the well known theorem \cite{zar-sam} the ring $\C[\mathrm e^{z\L},z]$ is also Noetherian as its quotient. 
\end{proof}

\subsection{Quasipolynomial actions of commutative finitely generated groups}

\begin{Lem}\label{lem:quasip}
If $G:\Z^p\times\C^q\to\Diff[[\C^d,0]]$ is a finitely generated commutative subgroup, then each Taylor coefficient of any orbit $G(t)^*f$, $f\in\C[[x]]$, is a quasipolynomial in the variables $t=(t_1,\dots,t_r)$ with the spectrum $\L\in\C^r$ which is a finitely generated lattice depending only on the linearizations of generators of the group $G$.
\end{Lem}

Arranging the monomials $x^\alpha\in\C[[x]]$ in any special order, say, \texttt{deg-lex} (see below), we can identify each element of the group $G(t)$ with a bi-infinite matrix whose entries (matrix elements) depend on $t$. The lemma asserts that all matrix elements of $G(t)$ depend on $t$ quasipolynomially.

\begin{Rem}\label{rem:lattice-desc}
The proof of the Lemma gives an explicit description of the spectrum $\L$ as follows. For each $i=1,\dots,r=p+q$ we define the subset $\L_i\subset\C$ (the ``partial lattice'' associated with the variable $z_i$) as follows:
\begin{enumerate}
 \item if $i\le p$ and the corresponding self-map $F_i$ has the form $F_i(x)=M_i x+\cdots$ (the dots stand for the higher order terms and $M_i$ as an invertible matrix with the eigenvalues $\mu_1,\dots,\mu_d$), then $\L_i$ is the lattice in $\C$ generated by $2\pi\mathrm i$ and the logarithms $\l_i=\ln \mu_i$;
 \item if $p<i<q$ and the corresponding vector field $v_i$ has the form $v_i(x)=H_ix+\cdots$, then $\L_i$ is the lattice generated by the eigenvalues $\l_1,\dots,\l_d$ of the linearization matrix $H_i$.
\end{enumerate}
Then the spectrum $\L\subseteq\C^r$ of the quasipolynomial algebra associated with the group generated as in Lemma~\ref{lem:quasip}, is the Cartesian product 
\begin{equation}\label{Lambda}
 \L=\L_1\times\cdots\times\L_r.
\end{equation}
\end{Rem}

A special case of this lemma for $t\in\Z^1$ is proved in \cite{a-miln}. We give a different proof which adapts easily to the case of flows and general commutative groups.

\begin{proof}[Proof of the Lemma]
Since the assertion concerns only finite order Taylor coefficients, one can replace the algebra $\C[[x]]$ by that of truncated polynomials of sufficiently high degree: the matrix element of the algebra automorphism $G(t)^*$ at the $(\alpha,\beta)$-position, which is coefficient before $x^\beta$ in the formal series $G(t)^* x^\alpha$, can be determined by looking at the truncated action of $G(t)^*$ on any $\C_m[x]$ with $m\ge\max(|\alpha|,|\beta|)$.
 
The algebra $\C_m[x]$ is finite-dimensional over $\C$, and the action of the group is by powers of the discrete generators and exponentials of the derivations. We have to describe the dependence of the matrix elements of the product \eqref{t} when $F_i^*$ are automorphisms and $V_j$ derivations of the finite-dimensional algebra of the truncated polynomials.

We start first with the cyclic (one-parametric) case. The mere fact that matrix elements of a one-parametric subgroup of finite-dimensional linear automorphisms, either discrete $\{L^t\}_{t\in\Z}$ or continuous $\{L^t=\exp tV\}_{t\in\C}$, are quasipolynomials, is well known. 

If $L=D+N$ is a finite-dimensional \emph{invertible} linear operator written as the sum of its commuting diagonal and nilpotent parts, $D=\operatorname{diag}(\mu_1,\dots,\mu_s)$, $N^{k+1}=0$ for some $k\le\dim L$ then for any $t\in\Z$ the binomial series for $(D+N)^t$ becomes a finite sum,
\begin{equation*}
 L^t=D^t+ t D^{t-1}N+\tfrac12{t(t-1)}D^{t-2}N^2+\cdots + \tfrac1{k!}t(t-1)\cdots(t-k+1)\,D^{t-k}N^k.
\end{equation*}
All matrix elements of the powers $D^{t-1},\dots,D^{t-k}$ as functions of $t$ are linear combinations of powers $\mu_1^t,\dots,\mu_s^t$, where $\{\mu_1,\dots,\mu_s\}$ are the eigenvalues of $L=L^1$. Therefore all matrix element of the powers $L^t$, $t\in\Z$, are quasipolynomials in one variable $t$ with the spectrum $\L=\{\l_1,\dots,\l_s\}\subset\C^1$, $\l_i=\ln \mu_i$. The logarithms are well defined, since by the invertibility of $L$, $\mu_i\ne 0$.

For the subgroup with an infinitesimal generator $V$, decomposed as $V=D+N$, $D=\operatorname{diag}(\l_1,\dots,\l_s)$, into the exponential series \eqref{exp} yields
\begin{equation*}
 \mathrm e^{tV}=\mathrm e^{tD}\cdot\mathrm e^{tN}=\mathrm e^{tD}\cdot(\text{matrix polynomial of }t),\qquad\forall t\in\C,
\end{equation*}
since the exponential series for $\mathrm e^{tN}$ is a finite (matrix) polynomial of degree $\le k$. Thus all matrix elements of the exponential $\mathrm e^{tV}$ are quasipolynomials with the spectrum $\L=\{\l_1,\dots,\l_s\}$.

To apply these arguments to the infinite-dimensional algebra $\C[[x]]$, we consider its finite-dimensional truncations. Let $L_m$ be the automorphism of the truncated polynomial algebra $\C_m[x]$ (resp., $V_m$ is a derivation of this algebra), obtained by the truncation of the formal automorphism $L\in\Aut\C[[x]]$, resp., the derivation $V:\C[[x]]\to\C[[x]]$. We claim that the the matrix elements of $L_m^t$, resp., $\mathrm e^{tV_m}$, are quasipolynomials with the spectra generated (as lattices) by the eigenvalues of the matrix $L_1$, resp., $V_1$, as explained above.

Without loss of generality we may assume that the local coordinates are chosen in such way that the linear part of the map (field) is lower-triangular, and use the corresponding order for the \texttt{deg-lex} ordering of the monomials. This means that:
\begin{enumerate}
 \item monomials $x^\alpha$ are ordered in the increasing order of their degrees $|\alpha|=\alpha_1+\cdots+\alpha_d$;
 \item monomials of the same degree are ordered lexicographically;
 \item the lexicographical order of the variables $x_1,\dots,x_d$ is such that $L^1 x_i=\mu_ix_i+ \cdots$, resp., $V_1 x_i=\l_ix_i+\cdots$, where the dots denote \texttt{lex}-inferior linear terms.
\end{enumerate}
By this choice we have
\begin{equation}\label{deglex}
\begin{aligned}
 Lx_i&=\mu_ix_i+\text{(\texttt{lex}-inferior linear terms)}+\text{(monomials of higher \texttt{deg})},
 \\
 Vx_i&=\l_ix_i+\text{(\texttt{lex}-inferior linear terms)}+\text{(monomials of higher \texttt{deg})}.
\end{aligned}
\end{equation}
For any monomial $x^\alpha$ we have, using the homomorphy of $L$ (resp., the Leibnitz rule for $V$) the identities
\begin{equation}\label{triang}
\begin{aligned}
 Lx^\alpha&=\mu^\alpha x^\alpha+\text{(monomials of higher \texttt{deg-lex} order)},
 \\
 Vx^\alpha&=\left<\l,\alpha\right> x^\alpha+\text{(monomials of higher \texttt{deg-lex} order)}.
\end{aligned}
\end{equation}
Here $\mu^\alpha=\mu_1^{\alpha_1}\cdots\mu_d^{\alpha_d}=\mathrm e^{\left<\l,\alpha\right>}$, $\left<\l,\alpha\right>=\l_1\alpha_1+\cdots+\l_d\alpha_d$ and $\l_i=\ln\mu_i$, $i=1,\dots,d$, as explained earlier.

Thus in the lexicographically ordered monomial basis each truncation $L_m$, resp., $V_m$, is triangular finite-dimensional operator with the spectrum formed by multiplicative (resp., additive) combinations of the eigenvalues of $L_1$ (resp., $V_1$) of total degree (resp., total multiplicity) $\le m$.

This observation shows that all matrix elements of $L_m^t$ (resp., $\mathrm e^{tV_m}$) are quasipolynomials in $t$ with the spectrum in the lattice $\L\subset\C$ generated by $\ln \mu_i$ (resp., $\l_i$).

% the Lie derivation $V$ is (block)-upper-triangular:
%\begin{equation*}
% V^*x^\alpha=\left<\l,\alpha\right>x^\alpha+(\text{linear combination of monomials of higher order}).
%\end{equation*}
%In other words, the spectrum of $V$ belongs to the lattice generated by the eigenvalues of the linearization matrix $H$.
%
%If $L=F^*$ is an automorphism associated with the (truncated) polynomial self-map $F(x)=Mx+\dots$, and the linearization matrix $M$ is lower-triangular with the eigenvalues $\mu_1,\dots,\mu_s$, then in a similar way
%\begin{equation*}
%F^* x^\alpha=\mu^\alpha x^\alpha+(\text{linear combination of monomials of higher order}),
%\end{equation*}
%where $\mu^\alpha=\mathrm e^{\left<\alpha,\l\right>}$. These arguments prove the Lemma for the case where $G$ is generated by one flow or iterations of one invertible self-map.

In the group having several commuting generators, the variable $t=(t_1,\dots,t_r)\in\Z^p\times\C^q\subseteq\C^r$ becomes multidimensional. By the above arguments, the matrix elements for the automorphism $G^*(t)$ depend on $t$ in such a way that dependence on each variable $t_i$ separately is quasipolynomial with a finitely generated spectrum $\L_i\subset\C$. One can easily see (e.g., by induction) that a function of several complex variables, quasipolynomial in each variable separately, is quasipolynomial in all variables. The corresponding spectrum $\L$ is the Cartesian product of finitely generated lattices \eqref{Lambda}, which obviously is itself a finitely generated lattice in $\C^r$.
\end{proof}

%\begin{Rem}
%One could avoid treating the integer and complex variables $t_i$ separately by simply embedding the group in the flow. The possibility of the embedding for a cyclic group with only one generator is explained in Remark~\ref{rem:embed}. One can show that a tuple of \emph{commuting} formal self-maps can be simultaneously embedded into a tuple of commuting flows, but this would lead us too far away.
%\end{Rem}

\section{Demonstration and discussion}

\subsection{Demonstration of Theorem~\ref{thm:main}}
Consider a commutative finitely generated group $G:\Z^p\times\C^q\to\operatorname{Diff}[[\C^d,0]]\simeq\operatorname{Aut}(\C[[x_1,\dots,x_d]])$, $p+q=r$, and two finitely generated radical ideals $I=\left<f_1,\dots,f_s\right>$, $J=\left<f_{s+1},\dots,f_n\right>$ in $\C[[x_1,\dots,x_d]]$, associated with the formal subvarieties $X$ and $Y$. From these data one can construct a family of ideals with marked (explicitly selected) generators
\begin{equation*}
 I_t=\left<G(t)^*f_1,\dots,G(t)^*f_s,f_{s+1},\dots,f_n\right>\subset\C[[x]], \qquad t\in\Z^p\times\C^q.
\end{equation*}
The ideal $I_t$ corresponds to the intersection of the formal varieties $g^{-1}(X)$ and $Y$ for $g=G(t)$. Our goal is to prove that the finite values of the codimension
\begin{equation*}
 \mu(t)=\codim_\C I_t, \qquad t\in\Z^p\times\C^q,
\end{equation*}
are bounded.

By Lemma~\ref{lem:quasip}, each coefficient $a_{i\alpha}=a_{i\alpha}(t)$ of each generator $G(t)^*f_i$ of $I_t$ is a quasipolynomial in $t$ with the common spectrum $\L$ defined only by the linear terms of the generators of $G$: $a_{i\alpha}(t)\in\C[\mathrm e^{t\L},t]$. 

Consider the increasing chain of polynomial ideals \eqref{chain}. Each ideal $\mathfrak I_m$ defines polynomial conditions imposed on the coefficients $a_{i\alpha}$ equivalent to the inequality $\codim I\ge m$. Substituting the explicit form of the generators, we obtain the \emph{quasipolynomial} conditions imposed on $I_t$ equivalent to the inequality $\mu(t)\ge m$. These quasipolynomial conditions together form the ideal $\mathscr I_m\subset\C[\mathrm e^{t\L},t]$, which depends on the group $G$ and the initial subvarieties $X,Y$, but in any case they form an increasing chain of ideals
\begin{equation}\label{chain-l}
 \mathscr I_1\subseteq\mathscr I_2\subset\cdots\subseteq\mathscr I_m\subseteq\cdots\subseteq \C[\mathrm e^{t\L},t],\ t\in\C^r.
\end{equation}
Note that strict inclusions between $\mathfrak I_m$ may become non-strict for the ideals $\mathscr I_m$.

%This dependence induces the morphism of $\C$-algebras
%\begin{equation*}
% \f:\mathfrak A\to\mathrm e^{t\L}[t], \qquad \mathfrak A=\C[\dots,a_{i\alpha},\dots],
%\end{equation*}
%uniquely determined by the values it takes on the generators $a_{i\alpha}$ of $\mathfrak A$. The morphism $\f$ depends on $G$ and the choice of the  generators of $I,J$.
%
% defining ideals of increasing codimensions in the infinitely generated algebra $\mathfrak A$ and its $\f$-image
%
%
%Assume that 

By Lemma~\ref{lem:noether}, the chain \eqref{chain-l} stabilizes after some finite number $m$:
\begin{equation*}
 \mathscr I_{m+1}=\mathscr I_{m+2}=\cdots=\mathscr I_{m+k}=\cdots\subseteq \C[\mathrm e^{t\L},t].
\end{equation*}
The values of $t$ which satisfy these (stable) quasipolynomial conditions correspond to the ideals $I_t$ of infinite codimension, $\mu(t)=+\infty$. This means that any \emph{finite} value of the codimension $\mu(t)$ cannot be bigger than $m$. \qed

\subsection{Generalizations}
The construction proving the main result, can be obviously modified to deal with multiple intersections between independently evolving subvarieties. For the sake of diversity we formulate this result for the holomorphic rather than formal dynamics (the $\C^\infty$-smooth version is also true for exactly the same reasons as before).

Let $G_1,\dots,G_k\subseteq\Diff(\C^d,0)$ be a tuple of commutative finitely generated subgroups of holomorphic self-maps, not necessarily disjoint, and $X_1,\dots,X_k$ a tuple of germs of analytic subvarieties of arbitrary dimensions. For any choice of elements $g_i\in G_i$ consider the multiplicity of the intersection between $g_i^{-1}(X_i)$ at the origin, defined as the codimension of the corresponding ideal in $\O$. This is a function $\mu=\mu(g_1,\dots,g_s)$ from $G_1\times\cdots\times G_k$ to $\N^*=\N\cup\{\infty\}$.

\begin{Thm}\label{thm:gener}
The finite values of the function $\mu$ are bounded.\qed
\end{Thm}

Theorem~\ref{thm:main} is a particular case of Theorem~\ref{thm:gener} for $s=2$, $G_1=G$ and $G_2=\{\operatorname{id}\}$. The proof remains literally the same.

\subsection{Non-commutative groups on $(\C^1,0)$}
The most interesting of all assumptions of Theorem~\ref{thm:main} is that of the commutativity of the group $G$: clearly, dropping the assumptions on finite generation immediately destroys the discussed phenomenon (consider, e.g., the action of the full group $\Diff(\C^d,0)$ on linear subspaces of appropriate dimensions). One can easily find examples of non-commutative groups for which the finite multiplicity intersections are unbounded.

\begin{Ex}\label{ex:nonsolv}
Let $G\subseteq\Diff(\C^1,0)$ be the subgroup generated by two non-identical holomorphic germs tangent to the identity:
\begin{equation*}
 G=\left<g_1,g_2\right>,\qquad g_i(x)=x+c_ix^{\nu_i+1}+\cdots, \quad c_i\ne 0,\ \nu_i\in\N, \ i=1,2.
\end{equation*}
One can immediately verify by the direct computation, that the commutator $g_3=[g_1,g_2]$ has a higher order of tangency to the identity, provided that $\nu_1\ne \nu_2$:
\begin{equation*}
 g_3(x)=x+(c_1c_2)(\nu_1-\nu_2)x^{\nu_1+\nu_2+1}+\cdots=x+c_3x^{\nu_3+1}+\cdots,\quad \nu_3=\nu_1+\nu_2,
\end{equation*}
see \cite{thebook}*{Proposition 6.11}. This implies that the group $G$ is non-solvable: together with the germ of order $\nu_3=\nu_1+\nu_2>\max\{\nu_1,\nu_2\}$ the commutator $[G,G]$ contains infinitely many germs of increasing finite orders, and this construction can be applied to all higher commutators as well.

Clearly, the multiplicity of the fixed point at the origin is unbounded, in contrast with the Shub--Sullivan theorem: if $G$ were cyclic generated by $g_1$, then  all non-identical elements will be of the form $g_1^n(x)=x+nc_1x^{\nu_1+1}+\cdots$ with the same multiplicity $\nu_1+1$ of the fixed point at the origin, hence $G\simeq\Z$.
\end{Ex}

Based on this example, one can construct (as described in the introduction) a non-solvable subgroup of $\Diff(\C^2,0)$ with the unbounded multiplicity of intersections.

For subgroups of $\Diff(\C^1,0)$ the question of maximal multiplicity of a periodic point at the origin in fact admits a complete solution, based on the known formal classification of such subgroups \cite{thebook}*{\parasymbol 6}.

\begin{Thm}\label{thm:alter}
Let $G=\left<g_1,\dots,g_n\right>\subseteq\Diff[[\C^1,0]]$ be a finitely generated subgroup of the group of formal 1-dimensional self maps. Then the finite values of the multiplicity of the periodic point at the origin are uniformly bounded from above if and only if the group $G$ is metabelian, i.e., its commutator $[G,G]$ is commutative.
\end{Thm}

\begin{proof}
The assertion of Theorem~\ref{thm:alter} follows immediately from the following version of the Tits alternative for one-dimensional self-maps, see \cite{thebook}*{Theorem 6.10}: any finitely generated subgroup of $\Diff[[\C^1,0]]$ is either metabelian or non-solvable (and contains a subgroup isomorphic to that described in Example~\ref{ex:nonsolv}). Thus in the non-metabelian case the intersection is unbounded.

A metabelian group is formally equivalent to a subgroup $G_\nu=\{c\mathrm e^{tV}:c\in\C^*, t\in\C\}$ generated by the linear self-maps $x\mapsto cx$ and the flow $\mathrm e^{tV}$ of a single formal vector field $V(x)=ax^\nu+\cdots$, $a\ne 0$, that is, a semidirect product of $\C^*\times\C$, see \cite{thebook}*{Theorem 6.22}. Clearly, the multiplicity of the periodic point at origin is the same number $\nu+1$ for all nontrivial elements of the group.
\end{proof}

\subsection{Some open problems}
As follows from Example~\ref{ex:nonsolv}, for non-solvable subgroups of $\Diff[[\C^d,0]]$ the function $\mu:G\to\N^*$ can exhibit unbounded growth of finite values. It is unclear whether what can be the growth rate of the function $\mu$ (e.g., how fast can it grow with the length of the words representing elements of the group grows to infinity). No meaningful examples were considered thus far, although one could reasonably expect that the growth will be fastest for groups with free subgroups and slowest for solvable.

It would also be interesting to upgrade the qualitative boundedness theorems for commutative subgroups to quantitative explicit estimates. For instance, if the germs $F_i:\C^d\to\C^d$ are polynomial of degree $\le \ell$ and the varieties $X,Y$ are algebraic of degree not exceeding the same (for simplicity) number $\ell$, then it is natural to look for an upper bound for multiplicities of the isolated intersections $F^n(X)\cap Y$ at the origin in terms of $\ell$ and $d$ (note that $F_i^{-1}$ need not be polynomial).

\section*{Appendix. The parametric case}

To address the parametric case, we assume that the generators $F_1,\dots,F_p$ and $v_{1},\dots,v_{q}$ of a group $G\subset\Diff[[\C^d,0]]$ depend analytically on finitely many parameters $\e=(\e_1,\dots,\e_\ell)$ varying in an arbitrarily small neighborhood of the origin in the parameter space, $\e\in(\C^\ell,0)$, so that each Taylor coefficient of each corresponding formal series depends analytically on $\e$. In the same way we allow the initial subvarieties $X=X(\e)$ and $Y=Y(\e)$ depend analytically on $\e$ in the sense that they are defined by functional equations $f_i=0$ analytically depending on $\e$, $f_i\in\mathscr O(\C^\ell,0)[[x]]$. Then (under some technical assumptions) the multiplicity of intersection becomes a function of $t\in\Z^p\times\C^q$ which additionally depends on the parameter $\e\in(\C^\ell,0)$ with values in $\N^*=\N\cup\{+\infty\}$. The conclusion is the same as in the non-parametric case: the finite values of the function $\mu$ are locally uniformly bounded: there exists a constant $m\in\N$ \emph{independent of $\e$}, such that all whenever $\mu(t,\e)<+\infty$, then $\mu(t,\e)\le m$ for any $t,\e$.

In order to be able to talk about ideals of infinite codimension for $\e\ne0$, we have to assume explicitly that the coefficients of the generators (initially defined as the germs from $\mathscr O(\C^\ell,0)$) admit representatives defined in some common neighborhood of the origin. The corresponding ring is smaller than $\mathscr O(\C^\ell,0)[[x]]$ but slightly larger than $\mathscr O(\C^\ell,0)\otimes_\C \C[[x]]$. In order to avoid technical difficulties, we will consider only the analytic case, so that the principal ring will be that of the germs $\mathscr O(\C^{\ell+d},0)$ of functions analytical in both $x$ and $\e$.

Very roughly, the idea of the proof is to replace the field of constants $\C$ by the field of meromorphic germs, the field of fractions of the ring $\mathscr O(\C^\ell,0)$, properly extended by finitely many ``algebraic germs''. The principal difficulty lies in the fact that the Jordan normal form is ill-depending on the parameters, thus the notion of a quasipolynomial has to be generalized. We briefly outline below the main changes required for this.

Assume as before, that
\begin{equation}
 F_i(\e,x)=M_i(\e)x+\cdots,\ i=1,\dots,p,\quad v_i(x)=H_i(\e)x+\cdots,\ i=1,\dots,q,
\end{equation}
where the $d\times d$-matrices $M_i(\e),H_i(\e)$ depend analytically on $\e\in(\C^\ell,0)$ and the dots stand for nonlinear terms.

Consider the $\C$-algebra of germs $\mathscr O(\C^\ell,0)$ and its field of fractions $\mathscr M(\C^\ell,0)$, the field of meromorphic germs. Consider the characteristic polynomials for the matrices $M_1(\e),\dots,M_p(\e),H_1(\e),\dots,H_q(\e)$, the monic polynomials with coefficients in $\mathscr O(\C^\ell,0)$. Let $\mathscr O_G(\C^\ell,0)$ be the ring obtained by adjoining the roots of all these polynomials to the ring $\mathscr O(\C^\ell,0)$, and denote by $\Bbbk_G$ or simply $\Bbbk$ the corresponding finite algebraic extension of the field $\mathscr M(\C^\ell,0)$. Elements of $\Bbbk_G$ can be considered as germs of algebraic functions of $\e$.

By the Jordan--Chevalley theorem \cite{humphreys}*{\parasymbol 4.2}, any matrix with entries in the ring $\mathscr O(\C^\ell,0)$ can be uniquely represented as a sum of two commuting terms (also with holomorphic entries), one semisimple and the other nilpotent. The semisimple part by definition is diagonalizable over the field $\Bbbk_G$ of algebraic germs. This forces to consider quasipolynomials over this larger field.

\begin{Def}
A quasipolynomial over the field $\Bbbk=\Bbbk_G$ as above is a \emph{finite} sum of the form 
\begin{equation}\label{pqp}
 q(z)=\sum \mathrm e^{\left<\boldsymbol\l_i,z\right>}p_i(z),\qquad z\in\C^r,\ \e\in(\C^\ell,0),
\end{equation} 
where:
\begin{enumerate}
 \item $\boldsymbol\l_i\in\Bbbk^r$ are $r$-tuples of elements from the field $\Bbbk$,
 \item $p_i\in\Bbbk[z]$ are polynomials in $z=(z_1,\dots,z_r)$ with coefficients in $\Bbbk$.
 \item the finite set $\L\subseteq\Bbbk^r$ of all exponentials $\boldsymbol\l_i$ is called the spectrum of the quasipolynomial.
\end{enumerate}
\end{Def}
Quasipolynomials can be considered as germs of (multivalued) functions on $(\C^\ell,0)\times\C^r$. 
If $\L\subseteq\Bbbk^r$ is a lattice, the set of all quasipolynomials over $\Bbbk_G$ with the spectrum in $\L$ is a $\C$-algebra $\Bbbk[\mathrm e^{z\L},z]$. 
 
After all these modifications are made, the following generalization of Lemma~\ref{lem:quasip} is proved by the same arguments as before, mutatis mutandis (cf.~with Remark~\ref{rem:lattice-desc}).

\begin{Lem}\label{lem:quasip-p}
If $G:\Z^p\times\C^q\to\Diff(\C^{d+\ell},0)$ is a finitely generated commutative subgroup of germs analytically depending on parameters and $\Bbbk=\Bbbk_G$ the corresponding finite algebraic extension of the field of meromorphic germs, then each Taylor coefficient of any orbit $G(t)^*f$, $f\in\mathscr O(\C^{d+\ell},0)$, is a quasipolynomial over $\Bbbk$ in the variables $t=(t_1,\dots,t_r)$ with the spectrum $\L\in\Bbbk^r$ which is a finitely generated lattice.\qed
\end{Lem}

By this lemma, one can in the same way as before construct from the universal chain of ideals \eqref{chain} the chain of ideals $\mathscr I_m$ in the corresponding quasipolynomial ring $\Bbbk[\mathrm e^{t\L},t]$, determining the condition that the corresponding multiplicity $\mu(t,\e)$ is greater or equal to $m$. Replacing rational expressions by polynomial, we finally obtain an ascending chain of ideals $\mathscr J_m$ in the ring $\mathscr O_G(\C^\ell,0)[\mathrm e^{t\L},t]$, whose length until full stabilization is an upper bound for the multiplicity of intersection.

But since the ring $\mathscr O(\C^\ell,0)$ is Noetherian, so is its finite extension $\mathscr O_G(\C^\ell,0)$, and by the same token as before, the chain of quasipolynomial ideals $\mathscr J_m$ must stabilize. This proves the following parametric form of Theorem~\ref{thm:main}.

\begin{Thm}\label{thhm:main-p}
If $G$ is a finitely generated commutative subgroup of $\Diff(\C^d,0)$ with generators analytically depending on finitely many parameters $\e\in(\C^\ell,0)$, and $X(\e),Y(\e)$ two germs of analytic subvarieties in $(\C^d,0)$ also depending analytically on these parameters, then the finite values of the multiplicity of intersection between $G(t)X$ and $Y$ are locally uniformly bounded.\qed
\end{Thm}

\end{document}